\documentclass[letterpaper,10pt,conference]{ieeeconf}

\usepackage{amsthm}  
\usepackage[smallerops]{newtxmath}
\usepackage[cal=cm,scr=boondoxo,bb=boondox,frak=euler]{mathalfa}
\let\labelindent\relax
\usepackage[inline]{enumitem}
\usepackage[utf8]{inputenc}
\usepackage{mathtools}
\usepackage{bm}
\usepackage{upgreek}
\usepackage{microtype}
\usepackage[subtle,mathspacing=normal]{savetrees}
\usepackage{tikz}
\usepackage{balance}
\usepackage{breqn}
\usepackage[ruled]{algorithm2e}
\usepackage{float}
\usepackage{comment}
\usepackage[style=ieee,backend=biber,doi=false,isbn=false]{biblatex}
\allowdisplaybreaks[1]
\IEEEoverridecommandlockouts
\definecolor{NU100}{RGB}{78,42,132}
\definecolor{NU90}{RGB}{91,59,140}
\definecolor{NU80}{RGB}{104,76,150}
\definecolor{NU70}{RGB}{118,93,160}
\definecolor{NU60}{RGB}{131,110,170}
\definecolor{NU50}{RGB}{147,128,182}
\definecolor{NU40}{RGB}{164,149,195}
\definecolor{NU30}{RGB}{182,172,209}
\definecolor{NU20}{RGB}{204,196,223}
\definecolor{NU10}{RGB}{228,224,238}
\definecolor{NUblue1}{RGB}{37,90,124}
\definecolor{NUblue2}{RGB}{61,122,161}
\definecolor{NUblue3}{RGB}{106,179,227}
\definecolor{NUyellow1}{RGB}{194,179,49}
\definecolor{NUyellow2}{RGB}{253,236,86}
\definecolor{NUyellow3}{RGB}{255,240,110}
\definecolor{NUred1}{RGB}{171,43,86}
\definecolor{NUred2}{RGB}{223,76,125}
\definecolor{NUred3}{RGB}{246,106,153}
\definecolor{gray80}{gray}{0.80}
\definecolor{gray90}{gray}{0.90}

\usetikzlibrary{calc,shapes,arrows,positioning,matrix}
\tikzset{block/.style={%
		text height=1.6ex,text depth=0.25ex,
		rectangle,
		minimum size=12mm,inner xsep=4mm,inner ysep=2mm,
		very thick,draw
}}
\tikzset{Lap/.style={%
		text height=1.6ex,text depth=0.25ex,
		rectangle,
		minimum size=12mm,inner xsep=4mm,inner ysep=2mm,
		very thick,draw
}}
\tikzset{sum/.style={%
		circle,
		minimum size=1mm,inner xsep=1mm,inner ysep=1mm,
		very thick,draw}}
\tikzset{link/.style={->,very thick,>=stealth,rounded corners}}
\tikzset{vertex/.style={draw,circle,very thick,black,fill=NUblue3!80}}
\tikzset{colvertex/.style={draw,circle,very thick,black,fill=NUred3!80}}

\makeatletter
\newcommand*{\tr}{%
	{\mathpalette\@tr{}}%
}
\newcommand*{\@tr}[2]{%
	\raisebox{\depth}{$\m@th#1\intercal$}%
}
\makeatother

\makeatletter
\let\save@mathaccent\mathaccent
\newcommand*\if@single[3]{%
	\setbox0\hbox{${\mathaccent"0362{#1}}^H$}%
	\setbox2\hbox{${\mathaccent"0362{\kern0pt#1}}^H$}%
	\ifdim\ht0=\ht2 #3\else #2\fi
}

\makeatother
\DeclareMathOperator*{\argmin}{arg\,min}

\newcommand*{\1}{\ensuremath{\mathbb{1}}}
\newcommand*{\La}{\ensuremath{\mathcal{L}}}
\newcommand*{\R}{\ensuremath{\mathbb{R}}}
\newcommand*{\n}{\mkern-1.5mu}

\theoremstyle{plain}
\newtheorem{theorem}{Theorem}

\theoremstyle{definition}
\newtheorem{remark}{Remark}
\AfterEndEnvironment{theorem}{\noindent\ignorespaces}
\AfterEndEnvironment{remark}{\noindent\ignorespaces}

\title{\LARGE \bf Self-Healing First-Order Distributed Optimization}

\author{Israel L.\ Donato Ridgley$^1$, Randy A.\ Freeman$^{1,3,4}$, and Kevin M.\ Lynch$^{2,3,4}$%
	\thanks{All authors are affiliated with Northwestern University, Evanston, IL 60208 USA (e-mail: israelridgley2023@u.northwestern.edu; freeman@northwestern.edu; kmlynch@northwestern.edu).}%
	\thanks{$^1$Department of Electrical \& Computer Engineering; $^2$Department of Mechanical Engineering; $^3$Northwestern Institute on Complex Systems; $^4$Center for Robotics and Biosystems}%
}

\begin{document}
	
	\maketitle
	\thispagestyle{empty}
	\pagestyle{empty}

	\begin{abstract}
		In this paper we describe a parameterized family of first-order distributed optimization algorithms that enable a network of agents to collaboratively calculate a decision variable that minimizes the sum of cost functions at each agent.  These algorithms are \emph{self-healing} in that their correctness is guaranteed even if they are initialized randomly, agents drop in or out of the network, local cost functions change, or communication packets are dropped.  Our algorithms are the first single-Laplacian methods to exhibit all of these characteristics. We achieve self-healing by sacrificing internal stability, a fundamental trade-off for single-Laplacian methods. 
	\end{abstract}
	
	\section{Introduction}
	\label{sec:intro}
	In this paper we study the distributed optimization problem, in which each agent in a network of $n$ agents calculates a decision vector that minimizes a global additive objective function of the form $f(\cdot)=\sum_i f_i(\cdot)$, where $f_i$ denotes the local convex objective function known only to agent $i$. Specifically, each agent maintains a local estimate $x_i$ of the global minimizer
	\begin{equation}
		x_{\text{opt}} = \argmin_{\theta} \sum_i f_i(\theta),
	\end{equation}
	which we assume is unique.
	The agents reach consensus $x_i = x_{\text{opt}}$ by computing the gradients of their local objective functions $\nabla\n f_i(x_i)$ and passing messages along the links of the communication network.
	
	Distributed optimization problems of this form have broad application. For example, a distributed set of servers or sensors could perform a learning task (e.g., classification) using their local data without uploading it to a central server for bandwidth, resiliency, or privacy reasons~\cite{forcangia10}. Swarms of robots can use distributed optimization to plan motions to solve the rendezvous problem \cite{rig08}. 
	
	The optimization of a collective cost function in a network setting has seen considerable interest over the last decade \cite{sunvanles20,shiqingwuyin15,lishiyan19,jak19,nedolsshi17,quli18,yuayinzhasay19p1,yuayinzhasay19p2}.
	Recently, several authors have adapted methods from control theory to study distributed optimization algorithms as linear systems in feedback with uncertainties constrained by integral quadratic constraints (IQCs) \cite{lesrecpac16,sunhules17,sunvanles20}. These works have made it possible to more easily compare the various known algorithms across general classes of cost functions and graph topologies.
	
	The work \cite{sunvanles20} uses these techniques to describe several recent distributed optimization algorithms within a common framework, then describes a new algorithm within that framework that achieves a superior worst-case convergence rate. However, all of the algorithms considered in \cite{sunvanles20}, including the authors' SVL algorithm, share a common undesirable trait:
	to reach the correct solution, their states must start in a particular subspace of the overall global state space and remain on it at every time step.
	If for any reason the state trajectories 
	leave this 
	subspace (e.g., incorrect initialization, dropped packets, 
	computation errors, agents leaving the network, changes to objective functions due to continuous data collection), then the system will no longer converge to the minimizer.
	Such methods cannot automatically recover from disturbances or other faults that displace their trajectories from this subspace; in other words, they are not \textit{self-healing}.
	
	In this paper, we extend our results from dynamic average consensus estimators \cite{ridfrelyn20,kiavancorfrelynmar19} to design a family of distributed optimization algorithms whose trajectories 
	need not evolve on a 
	pre-defined subspace.  We call such algorithms \emph{self-healing}.
	In practice, this means that our algorithms
	can be arbitrarily initialized,
	agents can join or leave the network at will, packets can be lost or corrupted, and agents can change their objective functions as necessary, such as when they collect new data. In order to handle the particular case of lost packets, we modify our algorithms with a low-overhead packet loss protocol; this modification is possible because our methods are self-healing.
	
	We refer to distributed optimization algorithms that communicate one or two variables (having the same vector dimension as the decision variable $x_i$) per time step as single- and double-Laplacian methods, respectively. Examples of single-Laplacian methods are SVL and NIDS, while examples of double-Laplacian methods are uEXTRA and DIGing \cite{sunvanles20,lishiyan19,jak19,nedolsshi17,quli18}. Our algorithms are the first self-healing single-Laplacian methods that converge to the exact (rather than an approximate) solution. They achieve self-healing by sacrificing internal stability, a fundamental trade-off for single-Laplacian methods. 
	In particular, each agent will have an internal state that grows linearly in time in steady state, but because such growth is not exponential it
	will 
	not cause any numerical instabilities 
	unless run over long time horizons. Double-Laplacian methods can achieve both internal stability and self-healing, but they require twice as much communication per time step 
	and converge no faster than single-Laplacian methods \cite{sunvanles20,kiavancorfrelynmar19}.
	
	\section{Preliminaries and Main Results}
	
		\subsection{Notation and terminology}
		We adopt notation similar to that in \cite{sunvanles20}.
		Let $\1_n$ be the $n$-dimensional column vector of all ones, $I_n$ be the identity matrix in $\R^{n \times n}$, and $\Pi_n = \frac{1}{n}\1 \1^\tr$ be the projection matrix onto the vector $\1_n$. We drop the subscript $n$ when the size is clear from context. We refer to the one-dimensional linear subspace of $\R^n$ spanned by the vector $\1_n$ as the \emph{consensus direction} or the \emph{consensus subspace}. We refer to the $(n-1)$-dimensional subspace of $\R^n$ associated with the projection matrix $(I_n - \Pi_n)$ as the \emph{disagreement direction} or subspace.
		  
		The variable $z$ represents the complex frequency of the $z$-transform.
		Subscripts denote the agent index whereas superscripts denote the time index. The symbol $\otimes$ represents the Kronecker product. $A^+$ indicates the Moore-Penrose inverse of $A$. Symmetric quadratic forms $x^\tr\n A x$ are written as $[\star]^\tr\n A x$ 
		to save space when $x$ is long. The local decision variables are $d$-dimensional and represented as a row vector, i.e., $x_i \in \R^{1 \times d}$, and the local gradients are a map $\nabla\n f_i : \R^{1 \times d} \rightarrow \R^{1 \times d}$. The symbol $||\cdot||$ refers to the Euclidean norm of vectors and the spectral norm of matrices.
		
		We model a network of $n$ agents participating in a distributed computation as a 
		weighted digraph $\mathcal{G} = (\mathcal{V},\mathcal{E})$, where $\mathcal{V} = \{1,...,n\}$ is the set of $n$ nodes (or vertices) and $\mathcal{E}$ is the set of edges such that if $(i,j) \in \mathcal{E}$ then node $i$ can receive information from $j$. We make use of the \textit{weighted graph} \textit{Laplacian} $\La \in \R^{n \times n}$ associated with $\mathcal{G}$ such that $-\La_{ij}$ is the weight on edge $(i,j)\in \mathcal{E}$, $\La_{ij}=0$ when $(i,j)\not\in\mathcal{E}$ and $i\neq j$, and the diagonal elements of $\La$ are 
		$\La_{ii} = -\sum_{j\neq i} \La_{ij}$,
		so that $\La \1 = 0$. 
		We define $\sigma = ||I-\Pi-\La||$, which is a parameter related to the edge weights and the graph connectivity.
		
	
		Throughout this work we stack variables and objective functions such that
		\[x^k = \begin{bmatrix}
			x_1^k \\
			\vdots \\
			x_n^k
		\end{bmatrix} \in \R^{n\times d} \;\; \text{and} \;\;
		\nabla\n F(x^k) = \begin{bmatrix}
			\nabla\n f_1(x_1^k) \\
			\vdots \\
			\nabla\n f_n(x_n^k)
		\end{bmatrix} \in \R^{n\times d}.\]
		
		\subsection{Assumptions}
		
		\begin{enumerate}[itemsep=0.25em,label=\textbf{(A\arabic*)},%
			align=left,leftmargin=*,series=assumptions]
			
			\item Given $0< m \leq L$, we assume that the local gradients are sector bounded on the interval $(m,L)$, meaning that they satisfy the quadratic inequality 
			\begin{equation*} \hspace*{-0.2in}
				[\star]^\tr
				\begin{bmatrix}
					-2mLI_d & (L+m)I_d\\
					(L+m)I_d & -2I_d
				\end{bmatrix}
				\begin{bmatrix}
					(x_i-x_{\text{opt}})^\tr\\
					(\nabla\n f_i(x_i)-\nabla\n f_i(x_{\text{opt}}))^\tr
				\end{bmatrix} \geq 0
			\end{equation*}
			for all $x_i \in \R^{1\times d}$, where $x_{\text{opt}}$ satisfies $\sum_{i=1}^n \nabla\n f_i(x_{\text{opt}}) = 0$. 
			We define the condition ratio as $\kappa = \frac{L}{m}$, which captures the variation in the curvature of the objective function. \label{a:1}
			
			\item The graph $\mathcal{G}$ is strongly connected. \label{a:2}
			\item The graph $\mathcal{G}$ is weight balanced, meaning that $\1^{\n\tr}\n \La = 0$. \label{a:3}
			\item The weights of $\mathcal{G}$ are such that $\sigma = ||I-\Pi-\La|| < 1$. \label{a:4}
		\end{enumerate}

		\begin{remark}
			Assumption \ref{a:1} is known as a \textit{sector IQC} (for a more detailed description see \cite{lesrecpac16}) and is satisfied when the local objective functions are $m$-strongly convex with $L$-Lipschitz continuous gradients.
		\end{remark}
		
		\begin{remark}
			Throughout this paper we assume without loss of generality that the dimension of the local decision and state variables is $d=1$.
		\end{remark}
		
		\begin{remark}
			Under appropriate conditions on the communications network, the agents can self-balance their weights in a distributed way to satisfy \ref{a:3}; 
			for example, they can use a scalar consensus filter like push-sum (see Algorithm~12 in \cite{haddomcha18}).
		\end{remark}

		\subsection{Results}
		In the following sections we present a parameterized family of distributed, synchronous, discrete-time algorithms to be be run on each agent such that, under assumptions \ref{a:1}-\ref{a:4}, we achieve the following:
		
		\begin{description}[itemsep=0.25em]
			\item[Accurate convergence:] in the absence of disturbances or other faults, the local estimates $x_i$ converge to the optimizer $x_{\text{opt}}$ with a linear rate.
			\item[Self-healing:] the system state trajectories need not evolve on 
			a pre-defined
			subspace and will recover from events such as arbitrary initialization, temporary node failure, computation errors,
			or changes in local objectives.
			\item[Packet loss protocol:] if agents are permitted a state of memory for each of their neighbors, they can implement a packet loss protocol that allows computations to continue in the event communication is temporarily lost. This extends the self-healing of the network to packet loss in a way that is not possible if the system state trajectories are required to evolve on a pre-defined subspace.
		\end{description}
		
		First we present the synthesis and analysis of our algorithm along with its performance relative to existing methods. Then we demonstrate via simulation that our algorithm still convergences under high rates of packet loss.

	\section{Synthesis of Self-Healing\\ Distributed Optimization Algorithms}
		\subsection{Canonical first-order methods}
		As a motivation for our algorithms, we use the canonical form first described in \cite{sunvanles19} and later used as the SVL template \cite{sunvanles20}. When the communication graph is constant, many single-Laplacian methods 
		such as SVL, EXTRA and Exact Diffusion can be described in this form \cite{shiqingwuyin15,yuayinzhasay19p1,yuayinzhasay19p2,sunvanles19,sunvanles20}, which is depicted as a block diagram in Figure~\ref{fig:block}. Algorithms representable by the SVL template can also be expressed as a state space system $G$ in feedback with an uncertain and nonlinear block containing the objective function gradients $\nabla\n F(\cdot)$ and the Laplacian $\mathcal{L}$ shown in Figure \ref{fig:feedback}, where 
		\begin{gather}
			G = \left[ \begin{array}{c|c|c}
				A & B_u & B_v \\
				\hline
				C_x & D_{xu} & D_{xv} \\
				\hline
				C_y & D_{yu} & D_{yv}\end{array}\right] = \left[ \begin{array}{cc|c|c}
				1 & \beta & -\alpha & -\gamma \\
				0 & 1 & 0 & -1 \\ 
				\hline
				1 & 0 & 0 & -\delta \\
				\hline
				1 & 0 & 0 & 0\end{array}\right] \otimes I_n.
		\end{gather}
	    We would like to alert the reader to a small notational difference between our work and \cite{sunvanles20}: in this work, the variable $x$ is the input to the gradients and the variable $y$ is the input to the Laplacian, whereas in \cite{sunvanles20} $y$ is the input to the gradients and $z$ is the input to the Laplacian (we cannot use $z$ here because we already use it as the frequency variable of the $z$-transform).
	    
		Algorithms representable by the SVL template, and more broadly all existing first-order methods with a single Laplacian, require that the system trajectories evolve on 
		a pre-defined
		subspace. From our work with average consensus estimators \cite{ridfrelyn20,kiavancorfrelynmar19}, we know that these drawbacks arise from the positional order of the Laplacian and integrator blocks. When the Laplacian feeds into the integrator, the output of the Laplacian cannot drive the integrator state away from the consensus subspace, which leads to an observable but uncontrollable mode. If the integrator state is initialized in the consensus subspace, or it is otherwise disturbed there, the estimate of the optimizer will contain an uncorrectable error.
		 Switching the order of the Laplacian and integrator renders the integrator state controllable but causes it to become inherently unstable because the integrator output in the consensus direction is disconnected from the rest of the system. We exploit this trade-off to develop self-healing distributed optimization algorithms with only a single Laplacian.
		\begin{figure}
			\centering
			\begin{tikzpicture}[node distance=0.6,>=stealth]
				\node (sum) [sum] {};
				\node (sss) [below left=0.01cm of sum] {$-$};
				\node (Int1) [block, right=0.8cm of sum] {${\dfrac{1}{z-1}I_n}$};
				\node (X) [right=of Int1] {};
				\node (ylabel) [above=0cm of X] {$y^k$};
				\node (L) [block, right=of X] {${\mathcal{L}}$};
				
				\node (d) [block, below= 0.2cm of L] {${\delta I_n}$};
				\node (sum2) [sum, left=of d] {};
				\node (xlabel) [above left=0 cm of sum2] {$x^k$};
				\node (sss) [below left=0.01cm of sum2] {$-$};
				\node (grad) [block, left=of sum2] {${\alpha \nabla\n F(\cdot)}$};
				\node (sum3) [sum] at (grad-|sum) {};
				\node (au) [above right=0cm of sum3] {$\alpha u^k$};
				
				\node (gam) [block, below=0.2cm of d] {${\gamma I_n}$};
				\node (sum4) [sum] at (gam-|sum3) {};

				\node (Int2) [block, below=0.2cm of gam] {${\dfrac{1}{z-1} I_n}$};
				\node (beta) [block, left=1.3cm of Int2] {${\beta I_n}$};

				\draw [link] (sum) -- (Int1);
				\draw [link] (Int1) -- (L);
				\draw [link] (Int1) -- (Int1-|sum2) -- (sum2);
				
				\node (off1) [right=of d] {};
				\node (j1) at (off1|-L) {}; 
				
				\node (v) [above=0cm of j1] {$v^k$};
				
				\draw [link] (L) -- (L-|j1) -- (j1|-d) -- (d);
				\draw [link] (L) -- (L-|j1) -- (j1|-gam) -- (gam);
				\draw [link] (L) -- (L-|j1) -- (j1|-Int2) -- (Int2);
				
				\draw [link] (d) -- (sum2);
				\draw [link] (sum2) -- (grad);
				\draw [link] (grad) -- (sum3);
				\draw [link] (sum3) -- (sum);
				\draw [link] (gam) -- (sum4);
				\draw [link] (sum4) -- (sum3);
				\draw [link] (Int2) -- (beta);
				\draw [link] (beta) -- (beta-|sum4) -- (sum4);
				
			\end{tikzpicture}
			\caption{The SVL template from \cite{sunvanles20} for first-order, single-Laplacian distributed optimization.}
			\label{fig:block}
		\end{figure}
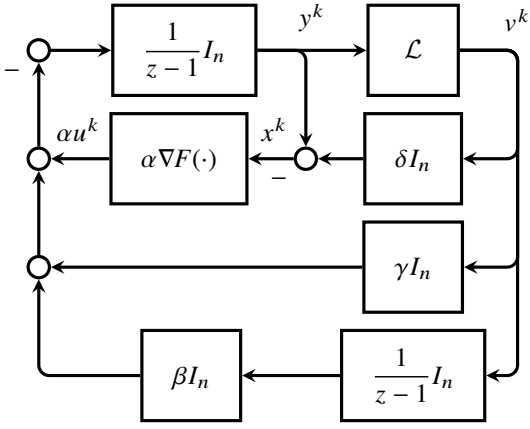
	
		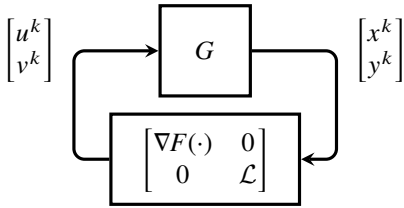
\begin{figure}
			\centering
			\begin{tikzpicture}
				\node (G) [block] {$G$};
				\node (NL) [block, below= 0.2cm of G] {${\begin{bmatrix} \nabla\n F(\cdot) & 0 \\ 0 & \mathcal{L} \end{bmatrix}}$};
				\node (off1) [left=of G] {};
				\node (inlabel) [left=0cm of off1] {${\begin{bmatrix}u^k\\v^k\end{bmatrix}}$};
				\node (off2) [right=of G] {};
				\node (outlabel) [right=0cm of off2] {${\begin{bmatrix}x^k\\y^k\end{bmatrix}}$};
				\draw [link] (G) -- (G-|off2) -- (off2|-NL) -- (NL);
				\draw [link] (NL) -- (NL-|off1) -- (off1|-G) -- (G);
			\end{tikzpicture}
			\caption{Distributed optimization algorithms 
			represented as a feedback interconnection of an LTI system $G$ and an uncertain block containing the gradients and
			the graph Laplacian.}
			\label{fig:feedback}
		\end{figure}
	
		\subsection{Factorization and integrator location}
		In the block diagram depicted in Figure \ref{fig:block}, it is unclear how to switch the Laplacian with the bottom integrator in a straightforward way. Instead we factor an integrator out of the $G(z)$ block of Figure \ref{fig:feedback},
		\begin{gather}
			\label{eq:Gz}
			G(z) = \begin{bmatrix}
				\dfrac{-\alpha}{z-1} & -\dfrac{\delta z^2 + (\gamma - 2\delta)z + (\beta+\delta-\gamma)}{(z-1)^2}\\[8pt]
				\dfrac{-\alpha}{z-1} & -\dfrac{\gamma z + (\beta - \gamma)}{(z-1)^2}
			\end{bmatrix}\otimes I_n\\
		\label{eq:factor}
		= \begin{bmatrix}
			\dfrac{-\alpha}{z-1} & -\dfrac{z-1+\zeta}{z-1} \\[8pt]
			\dfrac{-\alpha}{z-1} & \dfrac{-\gamma z - (\beta - \gamma)}{(z-1)(\delta z + \eta - \delta)}
		\end{bmatrix}
		\begin{bmatrix}
			1 & 0 \\[8pt]
			0 & \dfrac{\delta z + \eta - \delta}{z-1}
		\end{bmatrix}\otimes I_n,
		\end{gather}
        where
		\begin{gather}
		\eta = \gamma - \delta \zeta \;\;\text{and}\;\;\zeta = \begin{cases}
				\dfrac{\beta}{\gamma}, & \delta = 0\\
				\dfrac{\gamma - \sqrt{\gamma^2 - 4\beta \delta}}{2\delta}, & \text{otherwise.}
			\end{cases}
		\end{gather}
		 Swapping the order of the component matrices yields our new family of algorithms (where $G_{\n s}$ replaces $G$):
		\begin{align}
			G_{\n s}(z) &=
			\begin{bmatrix}
				1 & 0 \\[8pt]
				0 & \dfrac{\delta z + \eta - \delta}{z-1}
			\end{bmatrix}
			\begin{bmatrix}
				 \dfrac{-\alpha}{z-1} & -\dfrac{z-1+\zeta}{z-1} \\[8pt]
				 \dfrac{-\alpha}{z-1} & \dfrac{-\gamma z - (\beta - \gamma)}{(z-1)(\delta z + \eta - \delta)}
			\end{bmatrix}\otimes I_n\notag\\
			&=
			\begin{bmatrix}
				-\alpha \dfrac{1}{z-1} & -\dfrac{z-1+\zeta}{z-1} \\[8pt]
				-\alpha \dfrac{\delta z + \eta - \delta}{(z-1)^2} & -\dfrac{\gamma z + \beta - \gamma}{(z-1)^2}
			\end{bmatrix}\otimes I_n.
		\end{align}%
	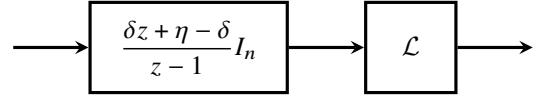
\begin{figure}
		\begin{tikzpicture}
			\node (L) [block] {$\La$};
			\node (int) [block, left=of L] {${\dfrac{\delta z + \eta - \delta}{z-1}}I_n$};
			\node (off1) [left=of int] {};
			\node (off2) [right=of L] {};
			
			\draw [link] (off1) -- (int);
			\draw [link] (int) -- (L);
			\draw [link] (L) -- (off2);
		\end{tikzpicture}
		\caption{The output of the integrator now feeds into the Laplacian, converting an uncontrollable and observable mode in the original SVL template to a controllable and unobservable one.}
		\label{fig:int_order}
	\end{figure}%
		Now the output of the integrator feeds directly into the Laplacian, as depicted in Figure \ref{fig:int_order}. We assume that our parameter choices satisfy
		\begin{align}
          \gamma^2 &\geq 4\beta\delta
		\end{align}
		so that the zeros of $G_{\n s}$ remain real and thus the system can be implemented with real-valued signals. The corresponding distributed algorithm is described in Algorithm~\ref{alg:1}, where $w_1$ and $w_2$ are the internal states of $G_{\n s}$, and the compact state space form is 
		\begin{equation}
			G_{\n s} = \left[ \begin{array}{cc|c|c}
				1 & 0 & -\alpha & -\zeta \\
				1 & 1 & 0 & -1 \\ 
				\hline
				1 & 0 & 0 & -1 \\
				\hline
				\delta & \eta & 0 & 0\end{array}\right] \otimes I_n.
		\end{equation}
		
		\begin{remark}
			The factorization in \eqref{eq:factor} is not unique; we chose it because it
			leads to a method still having only two internal states per agent.
			There may be other useful factorizations.
		\end{remark}
		
	\begin{algorithm}[t]
		\SetAlgoLined
		\LinesNotNumbered
		\KwSty{Initialization:} Each agent $i\in\{1,...,n\}$ chooses $w_{1i}^0,w_{2i}^0 \in \R^{1 \times d}$ arbitrarily. $\La \in \R^{n \times n}$ is the graph Laplacian.\\
		\For{$k=0,1,2,...$}{
			\For{$i\in\{1,...,n\}$}{
				\KwSty{Local communication}\\
				$y_i^k = \delta w_{1i}^k + \eta w_{2i}^k$\\[2pt]
				$v_i^k = \sum_{j=1}^n\mathcal{L}_{ij} y_j^k$\\
				\KwSty{Local gradient computation}\\
				$x_i^k = w_{1i}^k - v_i^k$\\[2pt]
				$u_i^k = \nabla\n f_i(x_i^k)$\\
				\KwSty{Local state update}\\
				$w_{1i}^{k+1} = w_{1i}^k - \alpha u_i^k - \zeta v_i^k$\\[2pt]
				$w_{2i}^{k+1} = w_{1i}^k + w_{2i}^k - v_i^k$
		}}
		\caption{\small Self-Healing Distributed Gradient Descent}
		\label{alg:1}
	\end{algorithm}
	
	\section{Stability and Convergence Rates Using IQCs}
		\subsection{Projection onto the disagreement subspace}
		As written, our family of algorithms is internally unstable. We use the projection matrix $(I-\Pi)$ to eliminate the instability from the global system without affecting $x^k$. This procedure is a centralized calculation that cannot be implemented in a distributed fashion, but it allows us to analyze the convergence properties of the distributed algorithm. 
		
		Consider the steady-state values $(w_1^\star,x^\star,u^\star,v^\star)$ and suppose $w_2^k$ contains a component in the $\1$ direction. Then that component does not affect the aforementioned values because it is an input to the Laplacian $\La$ (and lies in its nullspace); however, it grows linearly in time due to the $w_2$ update. Thus the system has an internal instability that is unobservable from the output of the bottom block in Figure~\ref{fig:feedback}. Since the component of $w_2^k$ in the consensus direction is unobservable to the variables $(w_1^k,x^k,u^k,v^k)$, we can throw it away without affecting their trajectories. Using the transformation $\hat{w}_2^k = (I - \Pi)w_2^k$, our state updates become
		\begin{align}
			\label{eq:up1}
			w_1^{k+1} &= w_1^k - \alpha u^k - \zeta v^k\\
			\label{eq:up2}
			\hat{w}_2^{k+1} &= (I-\Pi)w_1^k + (I-\Pi)\hat{w}_2^k - (I-\Pi)v^k \\
			\label{eq:up3}
			x^k &= w_1^k - v^k \\
			\label{eq:up4}
			\hat{y}^k &= \delta w_1^k + \eta \hat{w}_2^k \\
			u^k &= \nabla\n F(x^k) \\
			\label{eq:up6}
			v^k &= \mathcal{L} \hat{y}^k,
		\end{align}
		where $y^k$ was replaced with $\hat{y}^k$ in \eqref{eq:up4} and \eqref{eq:up6} to accommodate $\hat{w}_2^k$. These updates 
		lead to the state-space system
		\begin{equation}\label{eq:gp}
			G_m = \left[ \begin{array}{cc|c|c}
				I & 0 & -\alpha I & -\zeta I \\
				I-\Pi & I-\Pi & 0 & -(I-\Pi) \\ 
				\hline
				I & 0 & 0 & -I \\
				\hline
				\delta I & \eta I & 0 & 0\end{array}\right].
		\end{equation}
		\subsection{Existence and optimality of a fixed point}
		Now that we have eliminated the inherent instability of the global system, we can state the following about the fixed points:
		
		\begin{theorem}
			For the system described by $G_m$, there exists at least one fixed point $(w_1^\star,\hat{w}_2^\star, x^\star,\hat{y}^\star,u^\star,v^\star)$, and any such fixed point has $x^\star$ in the consensus subspace such that $x_i^\star = x_{\textnormal{opt}}$ for all $i \in \{1,\dots, n\}$, i.e., any fixed point of the system is optimal.
		\end{theorem}
		
		\begin{proof}
		First, assume that the fixed point $(w_1^\star,\hat{w}_2^\star, x^\star,\hat{y}^\star,u^\star,v^\star)$ exists. To prove that the variable $x^\star$ lies in the consensus direction, we show that $(I-\Pi)x=0$. From \eqref{eq:up2} and \eqref{eq:up3} we have that
		\begin{align}
			(I-\Pi)w_1^\star &= (I-\Pi)v^\star \\
			(I-\Pi)x^\star &= (I-\Pi)w_1^\star - (I-\Pi)v^\star\\
			&= 0.
		\end{align}
		Thus $x_i^\star = x_j^\star$ for all $i,j \in \{1,\dots,n\}$. Next we show that $x_i^\star = x_{\text{opt}}$. From \eqref{eq:up1} then plugging in \eqref{eq:up6}, we have
		\begin{align}
			-\alpha u^\star - \zeta v^\star &= 0 \\
			u^\star &= -\frac{\zeta}{\alpha} v^\star = -\frac{\zeta}{\alpha} \mathcal{L} \hat{y}^\star \\
			\1^\tr u^\star &= -\frac{\zeta}{\alpha} \1^\tr \mathcal{L} \hat{y}^\star \\
			\sum_{i=1}^n u_i^\star &= 0 \\
			\rightarrow \sum_{i=1}^n \nabla\n f_i(x_i^\star) &= 0\\
			\rightarrow x_i^\star &= x_{\text{opt}} \; \forall \; i \in \{1,\dots,n\}.
		\end{align}
		Thus any fixed point is optimal.
	
		Next, to construct a fixed point we define
		\begin{equation}
			\begin{aligned}
				x^\star &= \1 x_{\text{opt}}, & u^\star &= \nabla\n f(x^\star) \\
				v^\star &= -\frac{\alpha}{\zeta} u^\star, & w_1^\star &= x^\star + v^\star .
			\end{aligned}
		\end{equation}
		Then $\hat{w}_2^\star$ is the solution to the equation
		\begin{equation}
			\label{eq:w2sol}
			\zeta \eta \mathcal{L} \hat{w}_2^\star = -\alpha(I-\delta \mathcal{L})u^\star.
		\end{equation}
		Since $\hat{w}_2^k = \La^+ \La w_2^k$ (i.e., $\hat{w}_2^k$ is in the row space of $\La$), we write $\hat{w}_2^\star$ in closed form as
		\begin{equation}
			\hat{w}_2^\star = \frac{\alpha}{\zeta \eta}\La^+(\delta L - I) u^\star.
		\end{equation}
		Finally, setting $\hat{y}^\star = \delta w_1^\star + \eta \hat{w}_2^\star$ completes the proof.
		\end{proof}
	
		\begin{remark}
			If the graph is switching but converges in time such that the limit of the sequence of Laplacians exists, as with a weight balancer, then a solution to \eqref{eq:w2sol} still exists and an optimal fixed point can still be found. Furthermore, the proof techniques in the following section still hold for switching Laplacians (see \cite{sunvanles20} for more information).
		\end{remark}
		\subsection{Convergence}
		Following the approaches in \cite{lesrecpac16,sunhules17,sunvanles20}, we prove stability using a set of linear matrix inequalities. First we split our modified system from \eqref{eq:gp} into consensus and disagreement components. We define 
		\begin{align}
			A_m &= A_p \otimes \Pi + A_q \otimes (I-\Pi) \\
			B_{mu} &= B_{pu} \otimes \Pi + B_{qu} \otimes (I-\Pi)\\
			B_{mv} &= B_{pv} \otimes \Pi + B_{qv} \otimes (I-\Pi)
		\end{align}
		\begin{equation}
		\begin{aligned}
			A_p &= \begin{bmatrix}1 & 0 \\ 0 & 0\end{bmatrix}, & B_{pu} &= \begin{bmatrix}-\alpha\\0\end{bmatrix}, & B_{pv}&=\begin{bmatrix}-\zeta \\ 0\end{bmatrix}\\
			A_q &= \begin{bmatrix}1 & 0\\ 1 & 1\end{bmatrix}, & B_{qu} &= \begin{bmatrix}
				-\alpha \\0\end{bmatrix} & B_{qv} &=\begin{bmatrix} -\zeta\\ -1
			\end{bmatrix}.
		\end{aligned}
		\end{equation}
		We also define the matrices
		\begin{equation}
			M_0 = \begin{bmatrix}
				-2mL & L+m \\
				L+m & -2
			\end{bmatrix} \; \; \text{and} \; \; M_1 =
		\begin{bmatrix}
			\sigma^2-1 & 1\\1 & -1
		\end{bmatrix}.
		\end{equation}
		
		Notice that $M_0$ is associated with the sector bound from \ref{a:1} and that $M_1$ is associated with the $(1-\sigma,1+\sigma)$ sector bound on $\La$ with inputs from the disagreement subspace.
		
		We now make a statement analogous to Theorem 10 in \cite{sunvanles20}.
		\begin{theorem}
			If there exists $P,Q \in \R^{2\times 2}$ and $\lambda_0,\lambda_1 \in \R$, with $P,Q\succ 0$ and $\lambda_0,\lambda_1\geq 0$  such that
			\begin{gather}
				\label{eq:LMI1}
				[\star]^\tr
				\left[ \begin{array}{cc|c}
					P & 0 & 0\\
					0 & -\rho^2P & 0\\
					\hline
					0 & 0 & \lambda_0M_0
				\end{array}\right]
				\left[ \begin{array}{cc}
					A_p & B_{pu} \\
					I & 0 \\
					\hline
					C_{x} & D_{xu} \\
					0 & I
				\end{array}\right] \preceq 0, \\	
				[\star]^\tr
				\label{eq:LMI2}
				\arraycolsep=2.5pt
				\left[ \begin{array}{cc|c|c}
					Q & 0 & 0 & 0\\
					0 & -\rho^2Q & 0 & 0\\
					\hline
					0 & 0 & \lambda_0M_0 & 0 \\
					\hline
					0 & 0 & 0 & \lambda_1M_1 \\
				\end{array}\right]
				\left[ \begin{array}{ccc}
					A_q & B_{qu} & B_{qv} \\
					I & 0 & 0 \\
					\hline
					C_{x} & D_{xu} & D_{xv} \\
					0 & I & 0 \\
					\hline
					C_{y} & D_{yu} & D_{yv}\\
					0 & 0 & I
				\end{array}\right]\preceq 0,
			\end{gather}
		then the following is true for the trajectories of $G_m$:
			\begin{equation}\label{eq:main}
				\Bigg \| \begin{bmatrix}
					w_1^k - w_1^\star\\[.1cm] \hat{w}_2^k - \hat{w}_2^\star
				\end{bmatrix} \Bigg \| \leq \sqrt{\operatorname{cond}(T)}\rho^k \Bigg \| \begin{bmatrix}
				w_1^0 - w_1^\star\\[.1cm] \hat{w}_2^0 - \hat{w}_2^\star
			\end{bmatrix} \Bigg \|
			\end{equation}
			for a fixed point $(w_1^\star,\hat{w}_2^\star, x^\star,\hat{y}^\star,u^\star,v^\star)$, where $T = P \otimes I_n + Q \otimes (I_n - \Pi_n)$ and $\operatorname{cond}(T) = \frac{\lambda_{\textnormal{max}}(T)}{\lambda_{\textnormal{min}}(T)}$ is the condition number of $T$. Thus the output $x^k$ of Algorithm 1 converges to the optimizer with the linear rate $\rho$.
		\end{theorem}
		
		\begin{proof}
			Equation \eqref{eq:main} follows directly from Theorem 4 of \cite{lesrecpac16}. Since the states of $G_m$ are converging at a linear rate $\rho$, the rest of the signals in the system (including $x^k$) converge to the optimizer at the same rate. Additionally, the trajectories of $G_m$ and $G_{\n s}$ (Algorithm 1) are the same, save for $\hat{w}_2^k$ and $\hat{y}^k$, so $x^k$ in Algorithm 1 also converges to the optimizer with linear rate $\rho$. 
		\end{proof}
		 
		 To test the performance of our algorithm, we used the parameters $\beta = 0.5, \gamma = 1, \delta=0.5$. These parameters were inspired by the NIDS/Exact Diffusion parameters presented in \cite{sunvanles19}; however, we have done no work to find parameters that optimize the convergence rate. We then solved the LMIs \eqref{eq:LMI1} and \eqref{eq:LMI2} using Convex.jl \cite{convexjl} with the MOSEK solver \cite{mosek}, performing a bisection search on $\rho$ to find the minimum worst-case convergence rate for a given $\kappa$, $\sigma$, and $\alpha$. We used Brent's method from Optim.jl \cite{optim} to determine the optimal $\alpha$. We plot our results for $\kappa=10$ in Figure~\ref{fig:perf} and include the results for SVL (reproduced from \cite{sunvanles20}) for comparison. Our algorithm with these parameter choices achieves the same performance as NIDS for the NIDS parameter choice $\mu = 1$ as shown in \cite{sunvanles20}. The worst-case convergence rate of our algorithm is subject to the same lower bound, $\rho\geq \max (\frac{\kappa-1}{\kappa+1},\sigma) $, found in \cite{sunvanles20}.
		 
		 \begin{remark}
		 	We tested the convergence rates for our algorithm with Zames-Falb IQCs in place of Sector IQCs but saw no improvement. 
		 \end{remark}
		 
	\begin{figure}
		
		\includegraphics[width=0.5\textwidth]{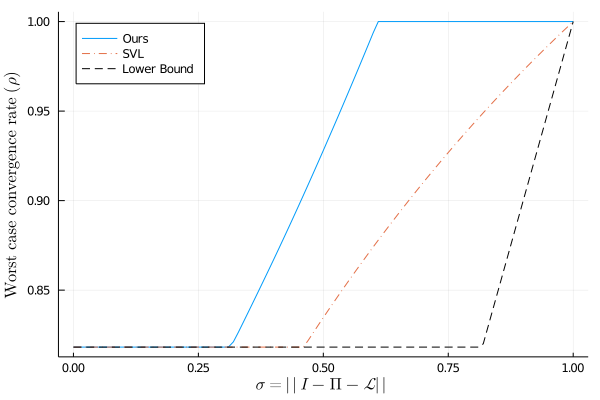}
		\caption{Performance of our algorithm compared with SVL for $\kappa = 10$ and $\sigma \in [0,1)$. Our NIDS-inspired parameter choices result in performance identical to that of NIDS in \cite{sunvanles20}. We have not made any attempts to choose ``optimal'' parameters like those of SVL.}
		\label{fig:perf}
	\end{figure}

	\section{Self-Healing Despite Packet Loss}
	\subsection{Packet loss protocol}
	We next give our agents some additional memory so that they can substitute previously transmitted values when a packet is lost. Each agent $i\in \{1,\dots,n\}$ maintains an edge state $e_{ij}^k$ for each $j\in \mathcal{N}_{\text{in}}(i)$ (the set of neighbors who transmit to $i$). Whenever agent $i$ receives a message from agent $j$, it updates the state $e_{ij}$ accordingly; however, if at time $k$ no message from neighbor $j$ is received, agent $i$ must estimate what would have likely been transmitted. One potential strategy is to substitute in the last message received, but because $y_j$ is growing linearly in quasi steady state, this naive strategy would ruin steady-state 
	accuracy. Instead we must account for the linear growth present in our algorithm, which we can do by analyzing the quantity $y_j^{k}-y_j^{k-1}$ at the quasi fixed point $(w_1^\star, x^\star,u^\star,v^\star)$:
	\begin{align}
	    y_j^{k}-y_j^{k-1} &= \delta(w_{1j}^\star-w_{1j}^\star) + \eta(w_{2j}^k - w_{2j}^{k-1})\\
	    &= \eta(w_{1j}^\star-v_j^\star)\\
	    &= \eta x_j^\star \approx \eta x_i^k
	\end{align}
	Therefore, when a packet is not received by a neighbor, agent $i$ scales its estimate of the optimizer and adds it to its previously received (or estimated) message. The packet loss protocol is summarized in Algorithm 2. By construction, the modifications included in Algorithm 2 will not alter the quasi fixed points of Algorithm 1, though we do not have a stability condition like Theorem 2 to present at this time. Instead, we show simulation evidence that Algorithm 2 does indeed converge, and packet loss does not appear to have a substantial impact on the convergence rate, even when the rate of packet loss is large. In the absence of dropped packets, the state trajectories of Algorithm 2 are equivalent to those of Algorithm 1.
	
	\begin{remark}
	    Algorithm 2 can be modified to include a forgetting factor $q$. If agent $i$ does not receive a packet from neighbor $j$ in $q$ time steps, then agent $i$ assumes that the communication link has been severed and clears $e_{ij}$ from memory.
	\end{remark}
	
	\begin{algorithm}[t]
		\SetAlgoLined
		\LinesNotNumbered
		\KwSty{Initialization:} Each agent $i\in\{1,...,n\}$ chooses $w_{1i}^0,w_{2i}^0 \in \R^{1 \times d}$ arbitrarily. $\La \in \R^{n \times n}$ is the graph Laplacian.  All $e_{ij}$ are initialized the first time a message is received from a neighbor.\\
		\For{$k=0,1,2,...$}{
			\For{$i\in\{1,...,n\}$}{
				\KwSty{Local communication}\\
				$y_i^k = \delta w_{1i}^k + \eta w_{2i}^k$\\[2pt]
				\For{$j\in \mathcal{N}_{\text{in}}(i)$}{
				    \eIf{Packet from $j$ received by $i$}{
				        $e_{ij}^k = y_{j}^k$
				        }
				        {$e_{ij}^k = \eta x_i^{k-1} + e_{ij}^{k-1}$
				    }}
				$v_i^k = \sum_{j=1}^n\mathcal{L}_{ij} e_{ij}^k$\\
				\KwSty{Local gradient computation}\\
				$x_i^k = w_{1i}^k - v_i^k$\\[2pt]
				$u_i^k = \nabla\n f_i(x_i^k)$\\
				\KwSty{Local state update}\\
				$w_{1i}^{k+1} = w_{1i}^k - \alpha u_i^k - \zeta v_i^k$\\[2pt]
				$w_{2i}^{k+1} = w_{1i}^k + w_{2i}^k - v_i^k$
		}}
		\caption{Packet loss protocol}
		\label{alg:2}
	\end{algorithm}

	\subsection{Classification example}
	
	To test the performance of our algorithm under packet loss, we solved a classification problem using the COSMO chip dataset \cite{cosmo} on an $n=7$ node directed ring lattice, shown in Figure \ref{fig:network}, such that $(i,j)\in \mathcal{E}$ when $j \in \{i+1,i+3,i+5\}\mod n$. All edge weights in the graph are set to 1/4 and $\sigma = ||I-\Pi-\La|| = 0.562$. We used the logistic loss function with $L_2$-regularization, yielding local cost functions
	
	\begin{equation}
	    f_i(x_i) = \sum_{j\in S_i} \log(1+e^{-l_jx_i^\tr M(d_j)}) + \frac{1}{n}||x_i||^2,
	\end{equation}
	where $S_i$ is the set of data indices local to agent $i$, $l_j$ is the label of data point $j$, and $M(d_j)$ is the higher-order polynomial embedding of data point $j$ (for more details see the logistic regression example in the COSMO github \cite{cosmo}). Using this cost, the corresponding sector bound $(m,L)$ is approximated as $m = \frac{2}{n}$ and
	\begin{gather}
	    L_i \leq \Big\|\frac{2}{n}I + \frac{1}{4}M_i^\tr M_i\Big\|, \; \;
	    L = \max_i L_i,
	\end{gather}
	where the rows of $M_i$ are $M(d_j)$ for $j\in S_i$.
	\begin{figure}
	    \centering
	    \begin{tikzpicture}
	        \node (0) {};
	        \node (1) at ($(0)+(0:1.8)$) [vertex] {$1$};
	        \node (2) at ($(0)+(-51.43:1.8)$) [vertex] {$2$};
	        \node (3) at ($(0)+(-102.86:1.8)$) [vertex] {$3$};
	        \node (4) at ($(0)+(-154.29:1.8)$) [vertex] {$4$};
	        \node (5) at ($(0)+(-205.71:1.8)$) [vertex] {$5$};
	        \node (6) at ($(0)+(-257.14:1.8)$) [vertex] {$6$};
	        \node (7) at ($(0)+(-308.57:1.8)$) [vertex] {$7$};

			\draw [link] (1) -- (2);
			\draw [link] (1) -- (4);
			\draw [link] (1) -- (6);
			\draw [link] (2) -- (3);
			\draw [link] (2) -- (5);
			\draw [link] (2) -- (7);
			\draw [link] (3) -- (4);
			\draw [link] (3) -- (6);
			\draw [link] (3) -- (1);
			\draw [link] (4) -- (5);
			\draw [link] (4) -- (7);
			\draw [link] (4) -- (2);
			\draw [link] (5) -- (6);
			\draw [link] (5) -- (1);
			\draw [link] (5) -- (3);
			\draw [link] (6) -- (7);
			\draw [link] (6) -- (2);
			\draw [link] (6) -- (4);
			\draw [link] (7) -- (1);
			\draw [link] (7) -- (3);
			\draw [link] (7) -- (5);
	    \end{tikzpicture}
	    \caption{The directed network topology for the classification example. All edge weights are 1/4.}
	    \label{fig:network}
	\end{figure}
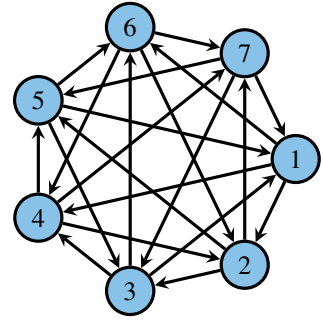
	
	Using $(m,L)$ and $\sigma$, we computed the optimal step size $\alpha$ for our algorithm using Brent's method and computed the SVL parameters as detailed in \cite{sunvanles20}. We then simulated both Algorithm 2 and SVL with and without packet loss and took the maximum error between the distributed algorithms and a centralized solution found using Convex.jl and MOSEK. We ran our algorithms using random initial conditions on the interval $[0,1]$ and the SVL algorithm using zero initial conditions. For the packet loss run of SVL,
	we held the previous message on each edge
	so that the fixed points would be unaffected. Packets had a 30\% chance of being lost, independent of each other. The results of these simulations are shown in Figure \ref{fig:packet_loss}. In this scenario, Algorithm 2 with lossy channels still converges to the optimum at a similar rate as Algorithm 1 with lossless channels, despite the high rate of packet loss that causes SVL to converge with high error. 
	\begin{figure}
		\includegraphics[width=0.5\textwidth]{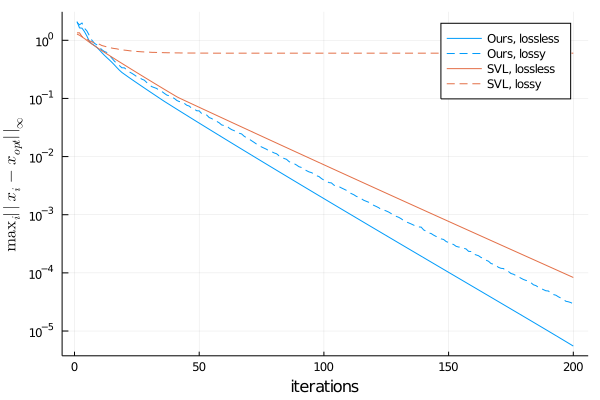}
		\caption{Simulation of Algorithm 1 and SVL in lossless channels as well as Algorithm 2 and SVL in lossy channels. The lossy channels are modeled with an independent 30\% packet loss. Error is the maximum error.}
		\label{fig:packet_loss}
	\end{figure}
	
	\section{Summary and Future Work}
	In this paper, we demonstrated the existence of a parameterized family of first-order algorithms for distributed optimization that do not require system trajectories to evolve on 
	a pre-defined
	subspace, despite having a single communicated variable. These algorithms are self-healing; they do not require the system to be initialized
	precisely
	and will recover from events such as agents dropping out of the network or changes to objective functions that might otherwise introduce uncorrectable errors. Furthermore, our algorithms can be augmented with our packet loss protocol, thereby allowing the system to converge to the optimizer even in the presence of heavily lossy communication channels. Our algorithms converge with a linear rate to the optimizer but contain an internal instability that grows linearly in time; however, this instability is unlikely to cause issues unless run over long time horizons.
	
	There is much left to investigate. We still need to consider the properties of other factorizations of $G(z)$ in \eqref{eq:Gz}, and possible factorizations of algorithms that are not subsumed by the SVL template. We need to explore the parameter space of the algorithm presented in this paper and, particularly, investigate if an optimization like that used to find the SVL parameters can be carried out. Finally, we will investigate a formal proof that Algorithm 2 still converges in the presence of packet loss.

	
	\balance
	\renewcommand*{\bibfont}{\footnotesize}
	\printbibliography
	
\end{document}